\documentclass[leqno,onefignum,onetabnum]{siamltex1213}
\usepackage{amsfonts}
\usepackage{amsmath,booktabs,ctable,threeparttable}
\usepackage{amssymb,amsfonts,boxedminipage}
\usepackage{algorithm}
\usepackage{algorithmic}
\usepackage{color,xcolor}
\usepackage{amsfonts,amsmath,amssymb}
\usepackage{booktabs,ctable,threeparttable}
\usepackage{graphicx,boxedminipage,appendix}
\usepackage{bm,enumerate,algorithm,algorithmic}
\usepackage[caption=false]{subfig}
\usepackage{multirow}
\usepackage{color}
\usepackage{mathrsfs}

\usepackage{amsthm}
\usepackage{lineno}
\usepackage{enumerate}
\begin{document}

\title{Theoretical and Computable Optimal Subspace Expansions
for Matrix Eigenvalue Problems\thanks{Supported by the
National Natural Science Foundation of China (Nos. 11771249 and 12171273).}}

\author{Zhongxiao Jia\thanks{Department of Mathematical Sciences,
Tsinghua University, Beijing 100084, People's Republic of China,
{\sf jiazx@tsinghua.edu.cn}.}}
\date{}
\maketitle

\newtheorem{Def}{Definition}
\newtheorem{Exm}{Example}
\newtheorem*{Prob}{The optimal subspace expansion problem}
\newtheorem{Thm}{Theorem}[section]
\newtheorem{Lem}{Lemma}[section]
\newtheorem{Cor}{Corollary}[section]
\newtheorem{Alg}{Algorithm}
\newtheorem{Exp}{Experiment}
\newtheorem{Ass}{Assumption}
\newtheorem{Rem}{Remark}
\newtheorem{Pro}{Proposition}

\newcommand{\Span}{\mathrm{span}}
\newcommand{\xp}{\mathbf{x}_\perp}
\newcommand{\Xp}{X_\perp}
\newcommand{\sep}{\mathrm{sep}}
\newcommand{\imag}{\mathrm{i}}

\begin{abstract}
Consider the optimal subspace expansion problem
for the matrix eigenvalue problem $Ax=\lambda x$: Which vector $w$ in the
current subspace $\mathcal{V}$, after multiplied by $A$,
provides an optimal subspace expansion for approximating a desired eigenvector
$x$ in the sense that $x$ has the smallest angle with the expanded
subspace $\mathcal{V}_w=\mathcal{V}+{\rm span}\{Aw\}$, i.e.,
$w_{opt}=\arg\max_{w\in\mathcal{V}}\cos\angle(\mathcal{V}_w,x)$?
This problem is important as many iterative methods
construct nested subspaces that successively
expand $\mathcal{V}$ to $\mathcal{V}_w$. An expression of $w_{opt}$ by Ye
(Linear Algebra Appl., 428 (2008), pp. 911--918) for $A$ general,
but it could not be exploited
to construct a computable (nearly) optimally expanded subspace.
He turns to deriving a maximization characterization
of $\cos\angle(\mathcal{V}_w,x)$
for a {\em given} $w\in \mathcal{V}$ when $A$ is Hermitian.
We generalize Ye's maximization characterization to
the general case and find its maximizer. Our main contributions consist of
explicit expressions of $w_{opt}$, $(I-P_V)Aw_{opt}$ and the optimally
expanded subspace $\mathcal{V}_{w_{opt}}$ for $A$ general, where $P_V$ is
the orthogonal projector onto $\mathcal{V}$.
These results are fully exploited to obtain computable optimally
expanded subspaces within the framework
of the standard, harmonic, refined, and refined harmonic Rayleigh--Ritz methods.
We show how to efficiently implement
the proposed subspace expansion approaches. Numerical experiments
demonstrate the effectiveness of our computable
optimal expansions.


\end{abstract}

\pagestyle{myheadings}
\thispagestyle{plain}
\markboth{ZHONGXIAO JIA}{OPTIMAL SUBSPACE EXPANSION FOR EIGENVALUE PROBLEMS}

\section{Introduction}
Consider the large scale matrix eigenproblem
$
Ax=\lambda x
$
with $A\in\mathcal{C}^{n\times n}$ and $\|x\|=1$, where $\|\cdot\|$ is the
2-norm of a vector or matrix. We are interested in an exterior
eigenvalue $\lambda$ and the associated eigenvector $x$. A number of
numerical methods have been available for solving this kind of
problem \cite{bai2000templates,parlett1998symmetric,
saad1992eigenvalue,vandervorst2002eigenvalue,stewart2001eigensystems}.
Of them, Lanczos or Arnoldi type methods \cite{bai2000templates,saad1992eigenvalue,stewart2001eigensystems}
are typical, and
they are projection methods on the sequence of nested
$k$-dimensional Krylov subspaces
$$
\mathcal{V}_k=\mathcal{K}_k(A,v_1)
={\rm span}\{v_1,Av_1,\ldots,A^{k-1}v_1\}
$$
and compute approximations to $\lambda$ and $x$.
For $A$ Hermitian, Arnoldi type methods
reduce to Lanczos type methods \cite{bai2000templates,parlett1998symmetric}.
Suppose that $V_k=(v_1,\ldots,v_k)$ is column orthonormal and is
generated by the Lanczos process in the Hermitian case
or the Arnoldi process in the non-Hermitian case,
and take the expansion vector $w_k=v_k\in \mathcal{V}_k$.
Then the next basis vector
$v_{k+1}$ is obtained by orthonormalizing $Aw_k$ against
$V_k$, the projection matrix $V_k^HAV_k$ of $A$ onto $\mathcal{V}_k$
is Hermitian tridiagonal for $A$ Hermitian or upper
Hessenberg for $A$ non-Hermitian, and the columns
of $(V_k,v_{k+1})$ form an orthonormal
basis of the $(k+1)$-dimensional Krylov
subspace $\mathcal{V}_{k+1}$, where the superscript $H$ denotes
the conjugate transpose of a vector or matrix.
If the expansion vector $v_k$ in $Av_k$
is replaced by any $w_k\in\mathcal{V}_k$ that is not deficient
in $v_k$, we deduce from Proposition 1 of Ye \cite{ye2008optimal}
that the same $\mathcal{V}_{k+1}$ and $V_{k+1}$ are generated.

However, if we start with a {\em non}-Krylov subspace $\mathcal{V}_k$,
then the expanded subspace $\mathcal{V}_{k+1}$ critically depends
on the choice of expansion vector $w_k\in\mathcal{V}_k$.
It is well known that the success of a general projection method requires
that an underlying subspace contain sufficiently accurate
approximations to $x$ but a sufficiently accurate subspace
cannot guarantee the convergence of the Ritz and harmonic
Ritz vectors obtained by the standard and harmonic
Rayleigh--Ritz methods with respect to the subspace
\cite{jia95,jia2001analysis}. To fix this deficiency,
the refined and refined harmonic Rayleigh--Ritz methods have been proposed,
which are shown to ensure the convergence of refined and refined
harmonic Ritz vectors computed by the refined and refined harmonic
Rayleigh--Ritz methods when the subspace is good
enough \cite{jia97,jia05,jia2001analysis,stewart2001eigensystems,wu2017}.

The accuracy of a projection subspace $\mathcal{V}$
for eigenvector approximation can be measured by its angle with the desired
unit length eigenvector $x$ \cite{parlett1998symmetric,saad1992eigenvalue,stewart2001eigensystems}:
\begin{equation}\label{angle}
\cos\angle(\mathcal{V},x)=\max_{z\in \mathcal{V}}\cos\angle(z,x)
=\max_{z\in \mathcal{V}}\frac{\mid x^Hz\mid}{\|z\|},
\end{equation}
where $\angle(z,x)$ denotes the acute angle between $x$ and the
nonzero vector $z$.
For a general $k$-dimensional non-Krylov subspace $\mathcal{V}$ with the subscript
$k$ dropped, in this paper we will consider the following optimal subspace
expansion problem:
{\em Suppose $\cos\angle(\mathcal{V},x)\not=0$, for any nonzero
$w\in\mathcal{V}$, multiply it by $A$, and define the $(k+1)$-dimensional expanded
subspace
\begin{equation}\label{expand}
\mathcal{V}_w=\mathcal{V}+{\rm span}\{Aw\}.
\end{equation}
Then which vector $w_{opt}\in \mathcal{V}$
is an optimal expansion vector that, up to scaling, solves
\begin{equation}\label{maxpro}
\max_{w\in\mathcal{V}}\cos\angle(\mathcal{V}_w,x)?
\end{equation}
}

This problem was first considered by Ye \cite{ye2008optimal}
and later paid attention by Wu and Zhang \cite{wu2014}.
The choice of $w\in \mathcal{V}$ is
essential to subspace expansion, and different $w$ may affect the quality
of $\mathcal{V}_w$ substantially. At expansion iteration $k>1$,
the Lanczos or Arnoldi type expansion takes $w=v_k$, the last column
of $V_k$. The Ritz expansion method \cite{wu2014,ye2008optimal}, a
variant of the Arnoldi method, uses the approximating Ritz vector
as $w$ to expand the subspace. It is mathematically equivalent to
the Residual Arnoldi (RA) method proposed in
\cite{lee2007residual,leestewart07}.

The Shift-Invert Residual Arnoldi (SIRA) method is an alternative of the RA
method for computing the eigenvalue $\lambda$ closest to a given target $\tau$. More
variants have been developed in \cite{jiali2014,jiali2015}, and
they fall into the categories of
the harmonic Rayleigh--Ritz, refined Rayleigh--Ritz,
and refined harmonic Rayleigh--Ritz methods
\cite{bai2000templates,parlett1998symmetric,saad1992eigenvalue,
stewart2001eigensystems,vandervorst2002eigenvalue}. The SIRA
type methods are mathematically equivalent to the counterparts of
the Jacobi--Davidson (JD) type methods; see Theorem~2.1 of
\cite{jiali2014}.
Just as the shift-invert Arnoldi (SIA) type methods applied
to $B=(A-\tau I)^{-1}$,
the SIRA type and JD type methods \cite{sleijpen2000jacobi} both
use $B$ to construct nested subspaces but, unlike the SIA type methods,
they project the
original $A$ rather than $B$ onto the subspaces
when finding approximations of $(\lambda,x)$ \cite{jiali2014}.
At iteration $k$, these methods and the SIA type methods
are mathematically common in solving a linear system
\begin{equation}\label{siainner}
(A-\tau I)u=w,
\end{equation}
i.e., computing $u=Bw$, where $w=v_k$ in the SIA type methods
and is mathematically equal to the current approximate
eigenvector in SIRA and JD type methods; see
\cite[Theorem 2.1 and its proof]{jiali2014} for details.

Ye \cite{ye2008optimal} finds an explicit expression of $w_{opt}$
for a general $A$, as stated below.

Ye's result \cite[Theorem 1]{ye2008optimal}:
{\em Let the columns of $V$ form an orthonormal basis of $\mathcal{V}$,
define the residual
\begin{equation}\label{Res}
R=AV-V(V^HAV),
\end{equation}
and assume that the rank $\rank(R)\geq 2$.
Then
\begin{equation}\label{expr1}
w_{opt}=VR^{\dagger}x+Vc
\end{equation}
for any $c\in \mathcal{N}(R)$, the null space of $R$, where the superscript
$\dagger$ denotes the Moore-Penrose generalized inverse of a matrix.
}

Unfortunately, Ye shows that (\ref{expr1})
cannot be exploited to obtain a computable (nearly) optimal
replacement of $w_{opt}$ as it involves the a priori $x$: (i) $Vc$ makes
no contribution to the expansion of $\mathcal{V}$ as $AVc=0$; (ii) since
$R^HV=0$, we must have $R^{\dagger}z=0$
when replacing $x$ by any $z\in\mathcal{V}$ and taking $c=0$,
leading to $w=VR^{\dagger}z=0$.
To this end, he gives up (\ref{expr1}) and instead turns to deriving a
maximization characterization of $\cos\angle(\mathcal{V}_w,x)$
for a {\em given} $w\in\mathcal{V}$ under the assumption that $A$ is
Hermitian; see Theorem 2 of \cite{ye2008optimal}
and the first result of Theorem \ref{thm:comVwx} to be presented in this paper.
But he does not find a solution to the maximization characterization problem
for a given $w\in \mathcal{V}$. Notice that one cannot
obtain $w_{opt}$ from the maximization characterization. Without
any other information, e.g., $w_{opt}$,
Ye has made an {\em approximate} analysis on the maximization characterization
of $\cos\angle(\mathcal{V}_w,x)$ and argued that the Ritz vector from
$\mathcal{V}$ might be a good approximate maximizer
of (\ref{maxpro}) and thus might be a nearly optimal expansion vector.
Unfortunately, as we shall see,
Ye's analysis and arguments have evident defects; see a detailed explanation after Remark 3 in this paper. We stress that Ye's proof of the maximization characterization
of $\cos\angle(\mathcal{V}_w,x)$ holds only for $A$ Hermitian.

Wu and Zhang \cite{wu2014}
do {\em not} consider the optimal expansion problem (\ref{maxpro}) itself. Instead,
they focus on showing that the refined Ritz vector
from $\mathcal{V}$ can be a better expansion vector than the Ritz vector
from $\mathcal{V}$ for $A$ general. Based on the analysis,
they have developed a refined variant of the RA method for a general $A$,
which has been numerically confirmed to be more efficient than the RA method.

In this paper, we shall revisit problem (\ref{maxpro})
for $A$ {\em general}. Our theoretical contributions consists of two parts.
The first part includes the generalization of Ye's major result in
\cite{ye2008optimal}, i.e., Theorem 2, to the non-Hermitian case.
We first prove that
$\cos\angle(\mathcal{V}_w,x)$ for a given $w\in\mathcal{V}$
can be formulated as a maximization characterization problem,
which extends  Theorem 2 of \cite{ye2008optimal} to the general case.
Then we find its maximizer.
This result is new even for $A$ Hermitian and can be exploited to
explain the defect of Ye's approximate analysis. These results are secondary,
and our major theoretical contributions are in the second part. We
establish informative expressions on $w_{opt}$,
$(I-P_V)Aw_{opt}$ and the optimally expanded $\mathcal{V}_{w_{opt}}$,
where $P_V=VV^H$ is the orthogonal projector onto $\mathcal{V}$.
The results show that (i) generally $w_{opt}\not=P_Vx$ where $P_Vx$ is the
orthogonal projection of $x$ onto $\mathcal{V}$
and $P_Vx/\|P_Vx\|$ is the best approximation to $x$
from $\mathcal{V}$, (ii) $(I-P_V)Aw_{opt}=RR^{\dagger}x$,
which is the orthogonal projection of $x$ onto the subspace
${\rm span}\{R\}$ and whose normalization $RR^{\dagger}x/\|RR^{\dagger}x\|$
is the best approximation to $x$ from ${\rm span}\{R\}$, (iii)
$\mathcal{V}_{w_{opt}}=\mathcal{V}\oplus {\rm span}\{RR^{\dagger}x\}$
with $\oplus$ the orthogonal direct sum, and
(iv) the orthogonal projection of $x$ onto $\mathcal{V}_{w_{opt}}$
is $(P_V+RR^{\dagger})x$.

For the purpose of practical computations, we make a clear and rigorous
analysis on the theoretical results and
obtain a number of computable optimally expanded subspaces
$\mathcal{V}_{\widetilde{w}_{opt}}$, which depend on chosen projection methods.
As has already seen from (\ref{expr1}) and the comments followed,
it is hard to interpret $w_{opt}$, let alone
a computable optimal replacement of $w_{opt}$.
Fortunately, we are able to take a completely new approach to consider a
computable optimal subspace expansion.
Our key observation is:
when expanding $\mathcal{V}$ to a computable optimal subspace,
it follows from the fundamental identity
\begin{equation}\label{id}
\mathcal{V}_{w_{opt}}=\mathcal{V}+ {\rm span}
\{Aw_{opt}\}=\mathcal{V}\oplus {\rm span}
\{(I-P_V)Aw_{opt}\}
\end{equation}
that seeking
a computable optimal replacement of $w_{opt}$ is
unnecessary. This identity provides us a new perspective
and motivates us to
find out a computable optimal replacement of
$(I-P_V)Aw_{opt}$ as a whole {\em rather than} $w_{opt}$ itself
by a certain computable optimal one. As it will be clear,
such an optimal replacement can
be defined precisely and is deterministic within
the framework of each of the standard, harmonic, refined,
and refined harmonic Rayleigh--Ritz methods, respectively. As it turns out,
computable optimal replacements are
the Ritz vector, refined Ritz vector, harmonic Ritz vector and
refined harmonic Ritz vector of $A$ from the subspace
${\rm span}\{R\}$ {\em rather than} $\mathcal{V}$
for these four kinds of projection methods.
Taking the standard Rayleigh--Ritz method as an example, we will
describe how to expand $\mathcal{V}$ to the computable optimal subspace
$\mathcal{V}_{\widetilde{w}_{opt}}$. We will also show that, based on
our results, there is some other novel optimal expansion
that is {\em independent} of the desired $x$
in practical computations and is promising and more robust.

The paper is organized as follows. In Section 2, we consider
the solution of the optimal subspace expansion problem (\ref{maxpro})
for a general $A$,
present our theoretical results, and make an analysis on them.
In Section 3, we show how to obtain computable
optimal replacements of $(I-P_V)Aw_{opt}$ and construct
optimally expanded subspaces $\mathcal{V}_{\widetilde{w}_{opt}}$.
We also present some other robust optimal expansion approach.
In Section 4, we report numerical experiments to demonstrate the effectiveness of
theoretical and computable optimal subspace expansion approaches, and compare
them with the Lanczos or Arnoldi type
expansion with $w=v_k$ and the the RA method, i.e., the
Ritz expansion \cite{wu2014,ye2008optimal} with $w$ being the Ritz vector from
$\mathcal{V}$.
In Section 5, we conclude the paper with some problems and issues that deserve
future considerations.

Throughout the paper, denote by $I$ the identity matrix
with the order clear from the context, by $\mathcal{C}^k$ the
complex space of dimension $k$, and by $\mathcal{C}^{n\times k}$
or $\mathcal{R}^{n\times k}$ the set of $n\times k$ complex
or real matrices.

\section{The optimal $w_{opt}$, $(I-P_V)Aw_{opt}$ and
$\mathcal{V}_{w_{opt}}$}
\label{optex}

For a general $A$, we first establish two results on $\cos\angle(\mathcal{V}_w,x)$
for a given $w$. The first characterizes it as a maximization problem and
generalizes Ye's Theorem 2 in \cite{ye2008optimal} and the second
gives a maximizer of this problem. We remind the reader
that these results are secondary.
After them, we will present our major results.

\begin{Thm}\label{thm:comVwx}
For $w\in\mathcal{V}$ with $x^Hw\neq0$ and $\|x\|=1$,
assume that $x\not\in\mathcal{V}$ and $Aw\not\in\mathcal{V}$,
define $r_w=(A-\phi I )w$ with $\phi=\frac{x^HA w}{x^Hw}$,
and let the columns of $Q_w$ be an orthonormal basis of the orthogonal
complement of ${\rm span}\{r_w\}$ with respect to $\mathcal{C}^n$.
Then 
\begin{equation}\label{cosVwx}
\cos\angle(\mathcal{V}_w,x)
=\max\limits_{b\in\mathcal{V}}
\frac{\cos\angle(b,x)}{\sin\angle\left(b,r_w\right)}
=\frac{\cos\angle(b_w,x)}{\sin\angle(b_w,r_w)},
\end{equation}
where
\begin{equation}\label{bopt}
b_w=V(Q_wQ_w^HV)^{\dagger}x=V(V^HQ_wQ_w^HV)^{-1}V^Hx.
\end{equation}
\end{Thm}

\begin{proof}
From the definition of $r_w$, we have
$$
\mathcal{V}+{\rm span}\{Aw\}=\mathcal{V}+{\rm span}\{r_w\}.
$$
By assumptions, we have $r_w\not=0$. Notice that any nonzero
$a\in\mathcal{V}_w$ but $a\not\in\mathcal{V}$ can be uniquely written as
$$
a=b+\beta r_w,
$$
where $b\in\mathcal{V}$ and $\beta$ is a nonzero scalar.
Since $Aw=r_{w}+\phi w$ and $x^Hr_{w}=0$, we obtain
\begin{eqnarray}
\cos\angle(\mathcal{V}_w,x)
&=&\max\limits_{a\in\mathcal{V}_w}\frac{|x^Ha|}
{\|a\|} = \max\limits_{b\in\mathcal{V},b+\beta r_w\neq 0}
\frac{|x^H(b+\beta r_w)|}{\|b+\beta r_w\|}\nonumber\\
&=&\max\limits_{b\in\mathcal{V},b+\beta r_w\neq 0}
\frac{|x^Hb|}{\|b+\beta r_w\|}\nonumber\\
&=&\max\limits_{b\in\mathcal{V}}\max\limits_{
\beta\not=0}\frac{\left|x^Hb\right|}{\|b+
\beta r_w\|}=\max\limits_{b\in\mathcal{V}}\frac{\left|x^Hb\right|}
{\min\limits_{\beta\not=0}\|b+\beta r_w\|}\nonumber\\
&=&\max\limits_{b\in\mathcal{V}}\frac{\left|x^H
b\right|}{\left\|b-\frac{r_w^Hb}{\|r_w\|^2}
r_w\right\|}=\max\limits_{b\in\mathcal{V}}
\frac{\left|x^Hb\right|}{\|b\|\sin\angle(b,r_w)}\nonumber\\
&=&\max\limits_{b\in\mathcal{V}}
\frac{\cos\angle(b,x)}{\sin\angle(b,r_w)}, \label{ratio}
\end{eqnarray}
which establishes the maximization characterization in (\ref{cosVwx}).

We now seek a maximizer of the maximization problem in (\ref{cosVwx}).
Since $x^H r_w=0$, by the definition of $Q_w$,
there exists a vector $z_w\in \mathcal{C}^{n-1}$ such
that the unit length eigenvector
$x=Q_w z_w$ with $\|z_w\|=1$. As a result, we obtain $Q_w^Hx=z_w$,
$Q_wQ_w^Hx=Q_wz_w=x$, and
$$
\cos\angle(b,x)
=\frac{|x^H b|}{\|b\|}
=\frac{|(Q_w z_w)^H b|}{\|b\|}
=\frac{|z_w^H(Q_w^H b)|}{\|b\|}.
$$
Since $\frac{\|Q_w^H b\|}{\|b\|}=\sin\angle(b,r_w)$,
it follows from the above that
\begin{eqnarray}
\cos\angle(b,x)
&=& \frac{\|Q_w^Hb\|}{\|b\|}\frac{|z_w^H(Q_w^H b)|}{\|Q_w^Hb\|}=
\sin\angle(b,r_w)\cos\angle(Q_w^Hb,z_w). \label{cossin}
\end{eqnarray}
Therefore, from (\ref{ratio}), (\ref{cossin}),
the orthonormality of $Q_w$ and $x=Q_wz_w$, writing a nonzero $b=V y$,
we obtain
\begin{eqnarray*}
\cos\angle(\mathcal{V}_w,x)
&=&\max\limits_{b\in\mathcal{V},b\not=0}
\cos\angle(Q_w^Hb,z_w)\\
&=&\max\limits_{b\in\mathcal{V},b\not=0}
\cos\angle(Q_wQ_w^Hb,Q_wz_w) \\
&=&\max\limits_{y\not=0}
\cos\angle(Q_wQ_w^HVy,x)\\
&=&\cos\angle((Q_wQ_w^HV)(Q_wQ_w^HV)^{\dagger}x,x)
\end{eqnarray*}
since $(Q_wQ_w^HV)(Q_wQ_w^HV)^{\dagger}x$ is the orthogonal
projection of $x$ onto the subspace ${\rm span}\{Q_wQ_w^HV\}$.
As a result, $y_w=(Q_wQ_w^HV)^{\dagger}x$ solves $\max\limits_{y\in\mathcal{C}^m,y\neq 0}
\cos\angle(Q_wQ_w^HVy,x)$, and $b_w=Vy_w$ is the first relation defined in (\ref{bopt}),
which proves the second relation in the right-hand side of (\ref{cosVwx}).

We next prove that $Q_wQ_w^HV$ has full column rank.
This amounts to showing that the cross-product matrix
$V^HQ_wQ_w^HV$ is positive definite by noting that $(Q_wQ_w^H)^2=Q_wQ_w^H$.
To this end, it suffices to prove that the solution of the homogenous
linear system $Q_wQ_w^HVz=0$ is zero.
Since $Q_wQ_w^H=I-\frac{r_wr_w^H}{\|r_w\|^2}$, the system
becomes
$$
\left(I-\frac{r_wr_w^H}{\|r_w\|^2}\right)Vz=0,
$$
which yields
\begin{equation}\label{nonsingular}
\frac{V^Hr_wr_w^HV}{\|r_w\|^2}z=z.
\end{equation}
Since
$$
\left\|\frac{V^Hr_wr_w^HV}{\|r_w\|^2}\right\|=\frac{\|V^Hr_w\|^2}
{\|r_w\|^2}=\cos^2\angle(\mathcal{V},r_w),
$$
taking norms in the two sides of (\ref{nonsingular}) gives
$$
\|z\|\cos^2\angle(\mathcal{V},r_w)\geq\|z\|,
$$
which holds only if $z=0$ or $\cos\angle(\mathcal{V},r_w)=1$. The latter
means that $r_w=Aw-\phi w\in \mathcal{V}$, i.e, $Aw\in\mathcal{V}$,
a contradiction to our assumption. Hence we must have $z=0$, and
$Q_wQ_w^HV$ has full column rank. Exploiting $Q_wQ_w^Hx=x$, we obtain
$$
y_w=(Q_wQ_w^HV)^{\dagger}x=(V^HQ_wQ_w^HV)^{-1}V^Hx,
$$
which proves the second relation in (\ref{bopt}).
\end{proof}

\begin{Rem}
This theorem holds for a general $A$.
For $A$ Hermitian, we have $\phi=\lambda$,
but for $A$ non-Hermitian or, more rigorously,
non-normal, we have $\phi\not=\lambda$. In the Hermitian case,
Theorem~2 in \cite{ye2008optimal} is the first
relation in the right-hand side of {\rm (\ref{cosVwx})}.
The second relation in the right-hand side of {\rm (\ref{cosVwx})}
is new even for $A$ Hermitian, and gives an explicit expression of
the maximizer $b_w$ of
the maximization characterization problem in {\rm (\ref{cosVwx})}.
\end{Rem}

\begin{Rem}
From {\rm (\ref{maxpro})}, relation {\rm (\ref{cosVwx})} shows that $w_{opt}$ solves
\begin{equation}\label{bmin}
\max_{w\in\mathcal{V}}\cos\angle(\mathcal{V}_w,x)=\max_{w\in\mathcal{V}}
\frac{\cos\angle(b_w,x)}{\sin\angle(b_w,r_w)}.
\end{equation}
However, because of the complicated nonlinear relationship
between $b_w$ and $r_w$,
it appears hard to solve the above problem
to derive an explicit expression of $w_{opt}$.
\end{Rem}

\begin{Rem}
By approximately maximizing
the first relation in the right-hand side of {\rm (\ref{cosVwx})} and
taking $b$ in it as the Ritz vector $z_1$ of $A$ from $\mathcal{V}$ that is used
to approximate the desired $x$, Ye \cite{ye2008optimal} makes
an approximate analysis on $\frac{\cos\angle(z_1,x)}{\sin\angle(z_1,r_w)}$,
and argues that, without any other information, $z_1$ may in practice
be a good approximate solution of
$\max_{w\in \mathcal{V}}\frac{\cos\angle(z_1,x)}{\sin\angle(z_1,r_w)}$. He thus
suggests $z_1$ as an approximation to the solution $w_{opt}$ of {\rm (\ref{maxpro})}.
However, Ye's proof of
the first relation in the right-hand side of {\rm (\ref{cosVwx})}
and his approximate analysis is unapplicable to a non-Hermitian $A$.
\end{Rem}

As a matter of fact, Ye's analysis in the Hermitian case
does not seem theoretically sound and has some arbitrariness,
causing that his claim may be problematic, as will be shown
below.

Since
$$
\cos\angle(\mathcal{V}_{w_{opt}},x)
\geq \cos\angle(\mathcal{V}_w,x)
$$ for any $w\in\mathcal{V}$,
Ye attempts to seek a good approximation to the maximizer $w_{opt}$
of $\max_{w\in\mathcal{V}}\cos\angle(\mathcal{V}_w,x)$.
To this end, in the first relation of the right-hand side in (\ref{cosVwx}),
Ye takes $b\in\mathcal{V}$ to be the Ritz vector $z_1$ of $A$ from
$\mathcal{V}$ that is used to approximate the desired $x$. Then
(\ref{cosVwx}) shows that
$$
\cos\angle(\mathcal{V}_w,x)=\frac{\cos\angle(b_w,x)}{\sin\angle(b_w,r_w)}
\geq \frac{\cos\angle(z_1,x)}{\sin\angle(z_1,r_w)},
$$
where $b_w$ is defined by (\ref{bopt}). Notice that
$b_w$ is a function of $w\in\mathcal{V}$.
However, since $z_1$ is a constant vector and
{\em independent} of $w\in\mathcal{V}$,
there is no reason that $z_1$ is close to $b_w$ unless $b_w\approx z_1$ for
some specific $w$ and $\mathcal{V}$, as we argue below.

Notice that $Q_wQ_w=I-\frac{r_wr_w^H}{\|r_w\|^2}$
and $r_w^Hx=0$.
Applying the Sherman--Morrison formula to $(V^HQ_wQ^HV)^{-1}$ and
making use of $r_w^HP_Vx=-r_w^H(I-P_V)x$,
by some manipulation we can justify that
\begin{equation}\label{bw}
b_w=P_V(x+\alpha_w w+\beta_wAw),
\end{equation}
where
$$
\alpha_w=\phi\frac{r_w^H(I-P_V)x}{\|(I-P_V)r_w\|^2},\
\beta_w=-\frac{r_w^H(I-P_V)x}{\|(I-P_V)r_w\|^2}
$$
with $\phi=\frac{x^HAw}{x^Hw}$.
Observe that $\|\alpha_w w\|$ and $\|\beta_wAw\|$
do not depend on the size of $\|w\|$ and are generally not small.
Suppose that $w$ is normalized.
In this case, $|\alpha_w|$ and $|\beta_w|$ are generally not small.
Recall that $P_Vx/\|P_Vx\|$ is
the best approximation to $x$ from $\mathcal{V}$ and $z_1$ is
the computable approximation to $x$ obtained by the standard Rayleigh--Ritz method.
Therefore, it is seen from (\ref{bw}) that
the constant vector $z_1$ is generally not a good approximation to
$b_w$ in direction for all $w\in\mathcal{V}$ unless $w\approx x$, which can occur
only when $P_Vx/\|P_Vx\|\approx x$, that is, $\cos\angle(\mathcal{V},x)$ is already sufficiently small. These mean that the functions
$$
\frac{\cos\angle(b_w,x)}{\sin\angle(b_w,r_w)} \ \ {\rm and}\ \
\frac{\cos\angle(z_1,x)}{\sin\angle(z_1,r_w)}
$$
generally have no similarity and their difference is not close to zero
unless $z_1\approx x$.
As a result, the maximizer of $\max_{w\in\mathcal{V}}
\frac{\cos\angle(z_1,x)}{\sin\angle(z_1,r_w)}$ generally bears no relation to
the maximizer $w_{opt}$ of
$$
\max_{w\in\mathcal{V}}\cos\angle(\mathcal{V}_w,x)
=\max_{w\in \mathcal{V}}\frac{\cos\angle(b_w,x)}{\sin\angle(b_w,r_w)}.
$$

The above analysis indicates that $b=z_1$
in the first relation of (\ref{cosVwx}) is generally not a good approximation to
$b_w$ and the resulting
$\max_{w\in\mathcal{V}} \frac{\cos\angle(z_1,x)}{\sin\angle(z_1,r_w)}$ is
generally not a good replacement of $\max_{w\in\mathcal{V}}\cos\angle(\mathcal{V}_w,x)$.
In the Hermitian case, by using a few approximate and heuristic
arguments, Ye \cite{ye2008optimal} argues that $w=z_1$ may be a good approximate
maximizer of $\max_{w\in\mathcal{V}}
\frac{\cos\angle(z_1,x)}{\sin\angle(z_1,r_w)}$ and then uses
$z_1$ as a replacement of $w_{opt}$ that solves (\ref{maxpro}).

In what follows we give up any possible further analysis
on Theorem~\ref{thm:comVwx} and
instead consider (\ref{maxpro}) from new perspectives.
We will establish a number of important and
insightful results. We derive a new
expression of $w_{opt}$, which is essentially the same as but formally different
from (\ref{expr1}) with $c=0$. The new form of $w_{opt}$
will play a critical role in establishing
explicit expressions of the a priori $(I-P_V)w_{opt}$, the
optimally expanded subspace $\mathcal{V}_{w_{opt}}$
and some other important quantities.

Note that $P_V=VV^H$ is the orthogonal projector onto $\mathcal{V}$.
Assume that $x\not\in\mathcal{V}$, $w\in\mathcal{V}$ and $Aw\not\in\mathcal{V}$.
Lemma 1 of \cite{ye2008optimal} states that
\begin{equation}
\cos\angle(\mathcal{V}_w,x)
=\sqrt{\cos^2\angle(\mathcal{V},x)+\cos^2\angle((I-P_V)Aw,x)},
\label{relationofvw and v}
\end{equation}
where $(I-P_V)Aw=r=Ry$ with $R$ defined by (\ref{res}) and
$w=Vy$. We should remind the reader that Lemma 1 of \cite{ye2008optimal}
uses the form $r=Ry$, but we use the different form
$(I-P_V)Aw$ here by writing
$R$ in the form of $(I-P_V)AV$. (\ref{id}) motivates us this change in form
and enables us to establish more informative
theoretical results on the optimal subspace expansion problem (\ref{maxpro}).
Relation (\ref{relationofvw and v}) indicates
that
\begin{equation}\label{sincos}
w_{opt}=\arg\max_{w\in\mathcal{V}}\cos\angle((I-P_V)Aw,x).
\end{equation}

Let the matrix $(V,V_{\perp})$ be unitary.
Then $I-P_V=V_{\perp}V_{\perp}^H$ and $(I-P_V)AV=V_{\perp}V_{\perp}^HAV$.
Define the vector
\begin{equation}\label{xperp}
x_{\perp}=(I-P_V)x,
\end{equation}
which is the orthogonal projection of $x$ onto the orthogonal complement of
$\mathcal{V}$ with respect to $\mathcal{C}^n$
and is nonzero for $x\not\in\mathcal{V}$. Then the matrix pair
\begin{equation}
\{V^HA^Hx_{\perp}x_{\perp}^HAV,\
V^HA^H(I-P_V)AV\}=\{V^HA^Hx_{\perp}x_{\perp}^HAV,\
V^HA^H V_{\perp}V_{\perp}^HAV\} \label{pair}
\end{equation}
restricted to the orthogonal complement $\mathcal{N}^{\perp}(V_{\perp}^HAV)$
of $\mathcal{N}(V_{\perp}^HAV)$ is Hermitian definite, that is,
the range restricted $V^HA^H(I-P_V)AV=V^HA^HV_{\perp}V_{\perp}^HAV$
is Hermitian positive definite. Note that
$$
\mathcal{N}^{\perp}(V_{\perp}^HAV)=\mathcal{R}(V^HA^HV_{\perp}),
$$
the column space of $V^HA^HV_{\perp}$. We write
the range restricted matrix pair (\ref{pair}) as
\begin{equation}
\{V^HA^Hx_{\perp}x_{\perp}^HAV,\
V^HA^HV_{\perp}V_{\perp}^HAV\}\big|_{\mathcal{R}(V^HA^HV_{\perp})}. \label{pair2}
\end{equation}

\begin{Thm}\label{optimal}
Assume that $x\not\in\mathcal{V}$, $w\in\mathcal{V}$ and $Aw\not\in\mathcal{V}$,
and let $R$ be defined by {\rm (\ref{Res})}.
Then 
the optimal expansion vector, up to scaling, is
\begin{equation}\label{exactwopt}
w_{opt}=V(V_{\perp}^HAV)^{\dagger}V_{\perp}^Hx=((I-P_V)AP_V)^{\dagger}x
=VR^{\dagger}x,
\end{equation}
\begin{eqnarray}
(I-P_V)Aw_{opt}&=&RR^{\dagger}x,\label{aw}\\
\mathcal{V}_{w_{opt}}&=&
\mathcal{V}\oplus {\rm span}\{RR^{\dagger}x\},\label{vwopt}
\end{eqnarray}
and the columns of $(V,RR^{\dagger}x/\|RR^{\dagger}x\|)$
form an orthonormal basis of $\mathcal{V}_{w_{opt}}$. Furthermore,
\begin{eqnarray}
\cos\angle(\mathcal{V}_{w_{opt}},x)&=&\cos\angle(P_Vx+RR^{\dagger}x,x),
\label{orthwopt}\\
\cos\angle(\mathcal{V}_{w_{opt}},x)&=&\cos\angle(\mathcal{V}_R,x),
\label{vwvr}
\end{eqnarray}
where the subspace
\begin{equation}\label{vr}
\mathcal{V}_R=\mathcal{V}\oplus {\rm span}\{R\}.
\end{equation}
\end{Thm}

\begin{proof}
Write $w=Vy$. Then by assumption, $\mathcal{V}$ is not an invariant subspace
of $A$, so that $(I-P_V)AV=V_{\perp}V_{\perp}^HAV\not=0$.
From (\ref{xperp}) and $(I-P_V)^2=I-P_V$, we obtain
\begin{eqnarray}
\cos^2\angle((I-P_V)Aw,x)
&=&\cos^2\angle((I-P_V)AVy,x) \notag\\
&=&\frac{|x^H(I-P_V)AVy|^2}{\|(I-P_V)AVy\|^2}\notag\\
&=&\frac{y^HV^HA^Hx_{\perp}x_{\perp}^HAVy}
{y^HV^HA^H(I-P_V)AVy},\label{cossq}
\end{eqnarray}
which is the Rayleigh quotient of the range
restricted Hermitian definite pair (\ref{pair2})
with respect to the nonzero vector $y$.
Keep in mind that $w_{opt}$ satisfies (\ref{sincos}).
Therefore, by the min-max
characterization of the eigenvalues of a Hermitian
definite pair (cf. e.g., \cite[pp. 281-2]{sunstewart}),
the solution to the maximization problem
\begin{equation}\label{maxprob}
\max_{w\in \mathcal{V}}\cos\angle((I-P_V)Aw,x)
=\max_{y}\cos\angle((I-P_V)AVy,x)
\end{equation}
is an eigenvector $y_{opt}$
of the range restricted definite pair (\ref{pair2}) associated with
its largest eigenvalue $\mu_{opt}$, and the optimal expansion
vector $w_{opt}=Vy_{opt}$ up to scaling.

Notice that $V^HA^Hx_{\perp}x_{\perp}^HAV\big|_{\mathcal{R}(V^HA^HV_{\perp})}$
is a rank-one Hermitian positive semi-definite matrix, and denote
\begin{equation}\label{bmatrix}
V^HA^HV_{\perp}V_{\perp}^HAV\big|_{\mathcal{R}(V^HA^HV_{\perp})}
=B\big|_{\mathcal{R}(V^HA^HV_{\perp})}.
\end{equation}
Therefore,
the range restricted Hermitian definite pair (\ref{pair2})
has exactly {\em one} positive eigenvalue, and the
other eigenvalues are {\em zeros}.
By (\ref{bmatrix}), the eigenvalues $\mu$
of the matrix pair (\ref{pair2}) are identical to those of
the rank-one Hermitian matrix
\begin{equation}\label{stand}
(B\big|_{\mathcal{R}(V^HA^HV_{\perp})})^{-\frac{1}{2}}
V^HA^Hx_{\perp}x_{\perp}^HAV (B\big|_{\mathcal{R}
(V^HA^HV_{\perp})})^{-\frac{1}{2}},
\end{equation}
and the eigenvectors $y$ of the pair (\ref{pair2}) are related to
the eigenvectors $\hat{y}$ of the above matrix by
\begin{equation}\label{y}
y=(B\big|_{\mathcal{R}(V^HA^HV_{\perp})})^{-\frac{1}{2}}\hat{y}.
\end{equation}
It is known that the unique positive eigenvalue $\mu_{opt}$ of
the matrix in (\ref{stand}) is
\begin{equation}\label{muopt}
\mu_{opt}=\|(B\big|_{\mathcal{R}(V^HA^HV_{\perp})})^{-\frac{1}{2}}V^HA^Hx_{\perp}\|^2
\end{equation}
and its associated (unnormalized) eigenvector is
$$
\hat{y}_{opt}=(B\big|_{\mathcal{R}(V^HA^HV_{\perp})})^{-\frac{1}{2}} V^HA^Hx_{\perp}.
$$
Substituting this into (\ref{y}) and $w_{opt}=Vy_{opt}$
gives the first relation in the right-hand side of (\ref{exactwopt}).
Exploiting $P_V=VV^H$ and $I-P_V=V_{\perp}V_{\perp}^H$, by the definition
of Moore--Penrose generalized inverse, we obtain the second relation in the
right-hand side of (\ref{exactwopt}).

By (\ref{Res}), the residual
\begin{equation}\label{res}
R=(I-P_V)AV,
\end{equation}
which shows that $(I-P_V)AP_V=RV^H$. As a result, we have
\begin{equation}\label{vrinv}
((I-P_V)AP_V)^{\dagger}=VR^{\dagger},
\end{equation}
which, together with the second relation in (\ref{exactwopt}), proves
\begin{equation}\label{wopt}
w_{opt}=VR^{\dagger}x,
\end{equation}
i.e., the third relation in (\ref{exactwopt}). Therefore, from (\ref{res})
and (\ref{wopt}),
we obtain
$$
(I-P_V)Aw_{opt}=(I-P_V)AVR^{\dagger}x=RR^{\dagger}x,
$$
which establishes (\ref{aw}).

Since
$$
\mathcal{V}_{w_{opt}}=\mathcal{V}+{\rm span}\{Aw_{opt}\}
=\mathcal{V}\oplus {\rm span}\{(I-P_V)Aw_{opt}\},
$$
it follows from (\ref{aw}) that (\ref{vwopt}) holds.

Relations (\ref{aw}) and (\ref{vwopt}) show that
the unit length $RR^{\dagger}x/\|RR^{\dagger}x\|$ and the columns of
$V$ form an orthonormal basis of the optimally expanded subspace
$
\mathcal{V}_{w_{opt}}=\mathcal{V}\oplus {\rm span}\{RR^{\dagger}x\},
$
and the orthogonal projector onto $\mathcal{V}_{w_{opt}}$ is
\begin{equation}\label{orthw}
P_V+RR^{\dagger}xx^HRR^{\dagger}/\|RR^{\dagger}x\|^2.
\end{equation}
Notice that $RR^{\dagger}$ itself is the orthogonal projector onto
${\rm span}\{R\}$. By right-multiplying the orthogonal projector
in (\ref{orthw}) with $x$ and exploiting $(RR^{\dagger})^2=RR^{\dagger}$,
it is easily justified that
the orthogonal projection of $x$ onto $\mathcal{V}_{w_{opt}}$
is $P_Vx+RR^{\dagger}x$, which proves
$$
\cos\angle(\mathcal{V}_{w_{opt}},x)=\cos\angle(P_Vx+RR^{\dagger}x,x),
$$
i.e., (\ref{orthwopt}) holds.

Note that $V^HR=0$ and the orthogonal projector onto $\mathcal{V}_R$ is
$P_V+RR^{\dagger}$. Therefore, $P_Vx+RR^{\dagger}x=(P_V+RR^{\dagger})x$
is the orthogonal projection of $x$ onto $\mathcal{V}_R$ defined
by (\ref{vr}), and from (\ref{orthwopt}) we obtain (\ref{vwvr}).
\end{proof}

\begin{Rem}
Since $\mathcal{V}_R\supset \mathcal{V}_{w_{opt}}$,
{\rm (\ref{orthwopt})} and {\rm (\ref{vwvr})} show that the best approximation
to $x$ from $\mathcal{V}_{w_{opt}}$ is identical to that from
the {\em higher} dimensional $\mathcal{V}_R$ containing $\mathcal{V}_{w_{opt}}$.
Remarkably, $\mathcal{V}_R$ is independent of $x$
but provides the same best approximation to $x$ as
$\mathcal{V}_{w_{opt}}$ does, and
all the other optimally expanded subspaces for eigenvectors $z$ of $A$
other than $x$ are also contained in $\mathcal{V}_R$ since $RR^{\dagger}z\in
{\rm span}\{R\}$. Therefore, unlike the a priori $\mathcal{V}_{w_{opt}}$
that favors $x$ only, $\mathcal{V}_R$
favor all the eigenvectors of $A$ equally, and
it is possible to use a suitable projection method to simultaneously compute
approximations to $x$ and other eigenvectors of
$A$ with respect to $\mathcal{V}_R$ rather than $\mathcal{V}_{w_{opt}}$.
We will come back to this point in the end of next section.
\end{Rem}

\begin{Rem}
It is seen from {\rm (\ref{wopt})} that the components of $R^{\dagger}x$ are the
coefficients of $w_{opt}$ in the basis $V$, while those of $V^Hx$ are the
coefficients of $P_Vx=VV^Hx$ in the basis $V$. Notice from {\rm (\ref{res})} that
$R^HV=0$. Then $R^{\dagger}V=0$ as $\mathcal{N}(R^{\dagger})=\mathcal{N}(R^H)$.
Relations $R^{\dagger}V=0$ and $V^HV=I$ show that $R^{\dagger}$ and $V^H$ are
different, so that $R^{\dagger}x$ and $V^Hx$ are not colinear generally.
As a result, taking $w=P_Vx$ and expanding $\mathcal{V}$ by
$(I-P_V)AP_Vx$ are not theoretically optimal in general.
In other words, although $P_Vx/\|P_Vx\|$ is the best approximation
to $x$ from $\mathcal{V}$,
relation {\rm (\ref{wopt})} shows that generally $w_{opt}\not=P_Vx$ up to scaling.
\end{Rem}

Finally, let us look at the particular case $\rank(R)=1$.
In this case, we can write $R$ as $R=\hat{v}u^H$
with $\|u\|=1$ and $\|R\|=\|\hat{v}\|$. Therefore,
\begin{equation}\label{scase}
V_{\perp}^HAV=V_{\perp}^HR=(V_{\perp}^H\hat{v}) u^H
\end{equation}
is a rank-one matrix. 
Let $w=Vy$. Exploiting (\ref{scase}),
$I-P_V=V_{\perp}V_{\perp}^H$, $I-P_V=(I-P_V)^2$, and
$\|(I-P_V)\hat{v}\|=\|V_{\perp}^H\hat{v}\|$, by some elementary
manipulation, we can prove that
\begin{eqnarray*}
\cos\angle((I-P_V)AVy,x)
&=&\cos\angle((I-P_V)\hat{v},x)
\end{eqnarray*}
for any $y$ satisfying $u^Hy\neq 0$.
Therefore, $w_{opt}=Vy$ for any $y$ satisfying $u^Hy\neq 0$,
and there are infinitely many $w_{opt}$'s.
However, $(I-P_V)Aw_{opt}=(I-P_V)\hat{v}(u^Hy)$
is unique
up to scaling for all $y$ satisfying $u^Hy\not=0$, and
$\mathcal{V}_{w_{opt}}$ is unique since
$$
(I-P_V)Aw_{opt}/\|(I-P_V)Aw_{opt}\|=(I-P_V)\hat{v}/\|(I-P_V)\hat{v}\|
$$
is unique and it, together with the columns of $V$, forms
an orthonormal basis of $\mathcal{V}_{w_{opt}}$.

\section{Computable optimal replacements of $(I-P_V)Aw_{opt}$ and
computable optimally expanded subspaces $\mathcal{V}_{\widetilde{w}_{opt}}$}

The results in Theorem~\ref{optimal} are a priori
and are not directly applicable to a practical expansion
of $\mathcal{V}$ as they involve the desired eigenvector $x$.
Therefore, for a non-Krylov subspace $\mathcal{V}$ with $\rank(R)\geq 2$,
one cannot use $w_{opt}$ in (\ref{expr1}), i.e.,
the third result of (\ref{exactwopt}), or a good approximation to it from
$\mathcal{V}$
to expand $\mathcal{V}$ in computations, as
we have elaborated in the introduction.
However, just as (\ref{id}) indicates, as far as the subspace
expansion is concerned,
it is $(I-P_V)Aw_{opt}$ rather than $w_{opt}$ itself that is used to expand $\mathcal{V}$. Therefore, instead of $w_{opt}$ itself,
the key is to consider $(I-P_V)Aw_{opt}$ as a whole when expanding $\mathcal{V}$.
This new perspective forms our basis of this section.

Theoretically, for a given subspace $\mathcal{V}$, the unit length vector
$P_Vx/\|P_Vx\|$ is the {\em theoretical} best approximation to $x$ from
$\mathcal{V}$. From the computational viewpoint, because of
$P_Vx/\|P_Vx\|\in \mathcal{V}$, projection methods are only
viable choices that seek {\em computable} best
approximations to $x$ from $\mathcal{V}$ using their
own extraction approaches. It has been commonly known from,
e.g., \cite{bai2000templates,parlett1998symmetric,saad1992eigenvalue,
stewart2001eigensystems,vandervorst2002eigenvalue}, that
the standard Rayleigh--Ritz, harmonic Rayleigh--Ritz, refined
Rayleigh--Ritz, and refined harmonic Rayleigh--Ritz methods
compute the Ritz, harmonic Ritz, refined Ritz, and refined harmonic
Ritz vectors of $A$ from a given subspace $\mathcal{V}$
and use them to approximate the desired $x$.
For the given $\mathcal{V}$, these approximate eigenvectors
are {\em computationally} the best approximations
to $x$ that these four kinds of projection methods can obtain,
and one cannot do better once a projection method is chosen. In other words,
these approximations are
computable optimal replacements of the a priori best approximation
$P_Vx/\|P_Vx\|$ to $x$ from $\mathcal{V}$ within
the framework of the four kinds of projection methods, respectively.
This naturally motivates us to introduce the following definition.

\begin{Def}\label{def}
For a chosen projection method, the approximation to $x$
extracted by it from $\mathcal{V}$ is called the computable optimal
replacement of $P_Vx/\|P_Vx\|$.
\end{Def}

For a chosen projection method, in terms of this definition,
for a chosen projection method,
we can fully exploit the theoretical results in Section 2
to expand $\mathcal{V}$ in its own computable optimal way.
Notice that $RR^{\dagger}$
is the orthogonal projector onto the subspace ${\rm span}\{R\}$ and $RR^{\dagger}x/\|RR^{\dagger}x\|$
is the theoretical best approximation to $x$ from ${\rm span}\{R\}$, and
write $v_{opt}=RR^{\dagger}x/\|RR^{\dagger}x\|$.
Then the columns of $(V,v_{opt})$ form an orthonormal basis
of $\mathcal{V}_{w_{opt}}$ in (\ref{vwopt}). Keep these results in mind. Then regarding the optimal expansion $v_{opt}$,
a chosen projection method obtains its computable best approximation
to $x$ from ${\rm span}\{R\}$, and we take it as
the computable optimal replacement of $v_{opt}$ to
$x$ from ${\rm span}\{R\}$ or, equivalently,
$(I-P_V)Aw_{opt}=RR^{\dagger}x$ in direction. We then use such a replacement,
to construct the corresponding computable optimally expanded subspace.

Specifically, by Definition~\ref{def},
for the standard Rayleigh--Ritz method, the harmonic Rayleigh--Ritz method,
the refined Rayleigh--Ritz method and the refined harmonic Rayleigh--Ritz methods,
the Ritz vector, the harmonic Ritz vector, the refined Ritz vector,
and the refined harmonic Ritz vector of $A$ from ${\rm span}\{R\}$ obtained by them
are the corresponding computable optimal replacements of the theoretical best
approximation $v_{opt}$ to $x$ from ${\rm span}\{R\}$.

In what follows we show how to efficiently and reliably achieve this goal
in computations.
Before proceeding, let us investigate (\ref{aw}) and get more insight
into the optimal subspace expansion. We present the following result.

\begin{Thm}
Assume that $(V,V_{\perp})$ is unitary. Then it holds that
\begin{equation}\label{rv}
\|RR^{\dagger}x\|\leq \|V_{\perp}V_{\perp}^Hx\|=\|(I-P_V)x\|.
\end{equation}
If $RR^{\dagger}x=V_{\perp}V_{\perp}^Hx$, then $x\in \mathcal{V}_{w_{opt}}$ and
the subspace expansion terminates.
\end{Thm}

\begin{proof}
Since $V^HR=0$, $V^HV_{\perp}=0$ and
$(V,V_{\perp})$ is unitary, we have ${\rm span}\{R\}\subseteq
{\rm span}\{V_{\perp}\}$, which proves (\ref{rv}) by noticing
that $RR^{\dagger}$ and $V_{\perp}V_{\perp}^H$ are the
orthogonal projectors onto ${\rm span}\{R\}$ and ${\rm span}\{V_{\perp}\}$,
respectively.

For $R\in\mathcal{C}^{n\times k}$, since ${\rm span}\{R\}\subseteq
{\rm span}\{V_{\perp}\}$, $\rank(R)=k_1\leq k$ and $\rank(V_{\perp})=n-k$,
we must have $\rank(R)<n-k$ when $k<n-k$.
On the other hand, if $k$ is sufficiently large
such that $k>n-k$, then $R$ must be rank deficient
as ${\rm span}\{R\}\subseteq{\rm span}\{V_{\perp}\}$ unconditionally.
Therefore, the theoretical optimal expansion $RR^{\dagger}x$ is part of
$(I-P_V)x$.
Notice that $x=P_Vx+(I-P_V)x$ and $P_Vx\in\mathcal{V}$,
Therefore,  if $RR^{\dagger}x=(I-P_V)x=V_{\perp}V_{\perp}^Hx$, then $x\in \mathcal{V}_{w_{opt}}$ already and
the subspace expansion terminates.
\end{proof}

Observe that, for all the afore-mentioned four kinds
of projection methods, the computable optimal replacements of
$v_{opt}$ are orthogonal to $\mathcal{V}$
because they lie in ${\rm span}\{R\}$ and $V^HR=0$.
Write the unit length approximate eigenvector obtained
by any chosen projection method applied to ${\rm span}\{R\}$
as $\widetilde{x}_R$, and take $\widetilde{v}_{opt}:=\widetilde{x}_R$.
Then the columns of $(V,\widetilde{v}_{opt})$ forms an orthonormal
basis of the computable optimally
expanded subspace, denoted by $\mathcal{V}_{\widetilde{w}_{opt}}$ hereafter,
where $\widetilde{w}_{opt}$ is a
corresponding replacement of the optimal expansion vector $w_{opt}$
and is not required in computations.

Mathematically, we can derive an explicit expression of
$\widetilde{w}_{opt}$ in terms of $\widetilde{v}_{opt}$, as the
following result shows.

\begin{Thm}
It holds that
\begin{equation}\label{tildewopt}
\widetilde{w}_{opt}=
V(V_{\perp}^HAV)^{\dagger}V_{\perp}^H\widetilde{v}_{opt}=VR^{\dagger}
\widetilde{v}_{opt}.
\end{equation}
\end{Thm}

\begin{proof}
Since
$\widetilde{w}_{opt}\in \mathcal{V}$, we have $P_V\widetilde{w}_{opt}=
\widetilde{w}_{opt}$, which leads to the equation
\begin{equation}\label{tildew}
(I-P_V)A\widetilde{w}_{opt}=(I-P_V)AP_V\widetilde{w}_{opt}
=\widetilde{v}_{opt}.
\end{equation}
It then follows from $((I-P_V)AP_V)^{\dagger}
=V(V_{\perp}^HAV)^{\dagger}V_{\perp}^H$
and (\ref{vrinv}) that (\ref{tildewopt}) holds.
\end{proof}

Compared with the expressions (\ref{exactwopt}) and (\ref{wopt})
of $w_{opt}$, we have replaced $V_{\perp}^Hx$ and
$R^{\dagger}x$ by $V_{\perp}^H\widetilde{v}_{opt}$
and $R^{\dagger}\widetilde{v}_{opt}$ in the expressions of $\widetilde{w}_{opt}$,
respectively. We should remind the reader that the solution $\widetilde{w}_{opt}$
to (\ref{tildew}) may not be unique and $\widetilde{w}_{opt}$ in
(\ref{tildewopt}) is the minimum 2-norm one.

Next we focus on some computational details on
the computable optimal subspace expansion approaches
that correspond to the four kinds of projection methods.

Suppose that $(V,\widetilde{v}_{opt})$ is available.
One goes to the next iteration, and
performs any one of the four projection methods. They all need to
form the matrix
\begin{equation}\label{projm}
(V,\widetilde{v}_{opt})^HA(V,\widetilde{v}_{opt})
\end{equation}
explicitly. This can be efficiently done in an updated way.
For brevity, hereafter suppose that all the quantities, such as $A$, $V$
and $\widetilde{v}_{opt}$, are all real when counting the computational cost.
For $V\in\mathcal{R}^{n\times k}$, the matrices $V^HAV$ and $AV$ are already
available at iteration $k$. To form the matrix in (\ref{projm}),
one needs to compute one matrix-vector $A\widetilde{v}_{opt}$,
and $2k+1$ vector inner products $V^H(A\widetilde{v}_{opt})$,
$\widetilde{v}_{opt}^H(AV)$ and $\widetilde{v}_{opt}^H(A\widetilde{v}_{opt})$.
The total cost is one matrix-vector product and $4nk+2n\approx 4nk$ flops.

The situation changes for other expansion vectors $w$, which include those used in
the Lanczos or Arnoldi type expansion approach, the Ritz expansion approach, i.e., the
mathematically equivalent RA method,
and the refined Ritz expansion approach.
When expanding $\mathcal{V}$ to $\mathcal{V}_w$, we need to compute $Aw$ first,
which costs one matrix-vector product, and then
to orthogonalize $Aw$ against $V$ to obtain
the unit length vector $v_w$, which costs approximately
$4nk$ flops when the (modified) Gram--Schmidt orthogonalization procedure
is used. In finite precision arithmetic, some sort of reorthogonalization
strategy may be required so as to ensure the numerical orthogonality
between $v_w$ and $V$. This will increase
the orthogonalization cost up to maximum $8nk$ flops when the complete
reorthogonalization is used. After $v_w$ is obtained,
one forms the projection matrix similar to that in (\ref{projm}),
where $\widetilde{v}_{opt}$ is replaced by $v_w$.
In contrast, once $\widetilde{v}_{opt}$
is available, the afore-described computable optimal subspace
expansion approach does not perform such a Gram--Schmidt orthogonalization
process and thus saves one matrix-vector product and $4nk\sim 8nk$ flops.

We now consider how to obtain
a computable optimal replacement $\widetilde{v}_{opt}$ in detail and count
its computational cost. First, we need to form the residual
$R$, which can be recursively
updated efficiently, as shown below.
Notice that $AV$ and $V^HAV$ are already available, and write
$R=R_k=(\widetilde{R}_{k-1},r_k)$
and $V=V_k=(V_{k-1},v_k)$ at iteration $k$. Then
\begin{eqnarray*}
\widetilde{R}_{k-1}&=&AV_{k-1}-V_{k-1}(V_{k-1}^HAV_{k-1})-v_k(v_k^HAV_{k-1})\\
&=&R_{k-1}-v_k(v_k^HAV_{k-1}),\\
r_k&=&Av_k-V_{k-1}(V_{k-1}^HAv_k)-(v_k^HAv_k)v_k,
\end{eqnarray*}
where $R_{k-1}$ is the residual at iteration $k-1$.
Since $R_{k-1}$ and the matrices in parentheses are
already available when forming $V_k^HAV_k$, the above updates of
$\widetilde{R}_{k-1}$ and $r_k$ approximately cost $2nk$ flops and $2nk$,
respectively, so that the total cost is approximately $4nk$ flops.

With $R$ available, we need to
construct an orthonormal basis matrix $Q$ of ${\rm span}\{R\}$,
so that ${\rm span}\{R\}={\rm span}\{Q\}$.
One can compute $Q$ in a number of
ways. For instance, the QR factorization with column pivoting, which costs
$4nkk_1 - 2k_1^2 (n + k) + 4k_1^3/3$ flops with $\rank(R)=k_1$,
approximately ranging from $8nk$ to $2nk^2$ flops
for the smallest $k_1=2$
and biggest $k_1=k$ \cite[pp. 302]{golub2013}, and the rank-revealing QR
factorization \cite{gu1996}, which costs nearly the same as
the QR factorization with column pivoting for $k\ll n$. The most reliable
but relatively expensive approach is the Chan R-SVD algorithm, which approximately
costs $6nk^2$ flops \cite[pp. 493]{golub2013}. Once $Q$ is computed,
we obtain
\begin{equation}\label{RQ}
RR^{\dagger}x=QQ^Hx.
\end{equation}

We point out that $\rank(R)=k_1<k$ is possible as $R$ may be (numerically)
rank deficient. 
Suppose that $\|R\|$ is not small, that is,
the columns of $V$ do not span a good approximate invariant subspace of $A$. Then
$R$ must be numerically rank deficient as one of the
$k$ Ritz pairs with respect to $\mathcal{V}$ converges to
some eigenpair, e.g., $(\lambda,x)$, as shown below:
Let $\widetilde{x}=Vy$ with $\|y\|=1$
be the Ritz vector approximating $x$ and $\mu$ be the
corresponding Ritz value, that is, $(\mu,y)$ is an
eigenpair of the projection matrix $V^HAV$. Then its associated
residual norm is $\|Ry\|=\|A\widetilde{x}-\mu\widetilde{x}\|$.
Whenever $\|Ry\|\leq tol$ with $tol$ sufficiently small,
$R$ must be numerically rank deficient since its smallest
singular value $\sigma_{\min}(R)\leq \|Ry\|\leq tol$.
Moreover, if some $k-k_1$ Ritz pairs converge, i.e., $\|Ry\|\leq tol$
for $k-k_1$ corresponding $y$'s, then
the numerical rank of $R$ equals $k_1$ since these $y$'s are linearly independent.

Indeed, we have observed in numerical experiments that
the numerical rank $k_1$ of $R$ satisfies $k_1\leq k-1$ and
becomes much smaller than $k$ as $k$ increases,
in the sense that $R$ has exactly $k-k_1$ tiny singular value(s) no more than
the level of $\|R\|\epsilon_{\rm mach}$ with $\epsilon_{\rm mach}$
being the machine precision.

With $Q$ at hand, for each of the four projection
methods, we can compute the corresponding
eigenvector approximation to $x$ from ${\rm span}\{Q\}$,
which is the corresponding computable optimal
replacement of $QQ^Hx/\|QQ^Hx\|$. Taking the standard
Rayleigh--Ritz method as an example, we need to form $Q^HAQ$,
whose eigenvalues are the Ritz values of $A$ with respect to ${\rm span}\{Q\}$.
Suppose that $Q$ is real, and notice that $Q\in\mathcal{R}^{n\times k_1}$.
Then we need $k_1$ matrix-vector products $AQ$ and $2nk_1^2$ flops to
compute $Q^HAQ$.
We compute the eigendecomposition of $Q^HAQ$ using the QR algorithm,
which costs approximately $25k_1^3$ flops for $A$ general and $9k_1^3$
flops for $A$ real symmetric. The desired Ritz vector
$\widetilde{v}_{opt}=Qy$ with $y$ being the unit length eigenvector of $Q^HAQ$
associated with its eigenvalue approximating
the desired eigenvalue $\lambda$.

In summary, suppose that the QR factorization with column pivoting
is used to compute $Q$. When computing $\widetilde{v}_{opt}$ and using it
to construct an orthonormal basis matrix $(V,\widetilde{v}_{opt})$
of $\mathcal{V}_{\widetilde{w}_{opt}}$,
at iteration $k$, we need $k_1$ matrix-vectors plus approximate
$2nk^2+2nk^2=4nk^2$ flops for $k_1\approx k$ and
$8nk_1+2nk_1^2\approx 2nk_1^2$ flops for $k_1\ll k$. Remarkably,
however, it is worthwhile to
point out that the $k_1$ matrix-vector products is the matrix-matrix
product $AQ$, which, together with $Q^H(AQ)$, should be computed
by using much more efficient level-3 BLAS operations.

For those subspace expansion approaches such as the Lanczos or
Arnoldi type expansion approach, and the Ritz expansion approach and the refined
Ritz expansion approach that are mathematically equivalent to
the RA method and the refined RA method,
when using the corresponding $w$ to
form $Aw$ and orthonormalize it against $V$ to expand
$\mathcal{V}$ to $\mathcal{V}_w$, we need one
matrix-vector product and $4nk\sim 8nk$ flops. Therefore,
the main cost of the $k$-step subspace expansion approach
is $k$ matrix-vector products and $2nk^2\sim 4nk^2$ flops. However,
successively updating $AV$ and
$V^H(AV)$ consists of level-2 BLAS and level-1 BLAS operations.
In contrast, since the matrix-matrix products
$AQ$ and $Q^HAQ$ are computed using
level-3 BLAS operations, the computable optimal
subspace expansion approach is much more
efficient than the $k$-step usual subspace expansion one even if
$k_1\approx k$.


Finally, let us return to (\ref{orthwopt}),
(\ref{vwvr}), and Remark 4 in Section 2. Relations
(\ref{orthwopt}) and (\ref{vwvr}) have proved that
the best approximations to $x$ from $\mathcal{V}_{w_{opt}}$
and $\mathcal{V}_R$ are the same and equal to
$(P_V+QQ^H)x/\|(P_V+QQ^H)x\|$ due to (\ref{RQ}). Previously, we have
proposed the expansion approach that obtains computable optimal
replacements of $v_{opt}=QQ^Hx/\|QQ^Hx\|$ from ${\rm span}\{R\}$ and
then expands $\mathcal{V}$ to $\mathcal{V}_{\widetilde{w}_{opt}}$,
from which we compute new approximate eigenpairs in the next
iteration. In the proposed approach, we first
compute $Q$, and then form $AQ$ and $Q^HAQ$. Nevertheless,
although the orthogonal projections of $x$ onto $\mathcal{V}_{w_{opt}}$
and $\mathcal{V}_R$ are identical, the computable optimal
approximations to $x$ from $\mathcal{V}_{w_{opt}}$
and $\mathcal{V}_R$ are {\em not} the same.
Furthermore, the approximate eigenpairs obtained by
the four kinds of projection methods applied to $\mathcal{V}_R$
are generally more accurate than the corresponding counterparts applied to
$\mathcal{V}_{w_{opt}}$ since $\mathcal{V}_{w_{opt}}\subset
\mathcal{V}_R$. The same conclusions also hold
when $\mathcal{V}_{w_{opt}}$ is replaced by $\mathcal{V}_{\widetilde{w}_{opt}}$
since $\mathcal{V}_{\widetilde{w}_{opt}}\subset \mathcal{V}_R$.

Regarding the cost, recall that the proposed computable
optimal expansion approaches have
computed $Q$, which, together with $V$, form an orthonormal basis
of $\mathcal{V}_R$. Therefore, we have already expanded
$\mathcal{V}$ to $\mathcal{V}_R$.
Naturally, this motivates us to compute possibly
better eigenpair approximations
of $(\lambda,x)$ with respect to the higher dimensional $\mathcal{V}_R$.
This results in a new subspace expansion approach. In the next iteration,
we need to form the projection matrix $(V,Q)^HA(V,Q)$. Since
$AV$ and $AQ$ are already available in the previously
proposed computable optimal expansion approaches, the extra cost is
the computation of $V^H(AQ)$ and $Q^H(AV)$, which can be
done using efficient level-3 BLAS operations.

More importantly, we can get more insight into the computable
optimal subspace expansion approaches.
They construct approximations $\mathcal{V}_{\widetilde{w}_{opt}}$'s to
the optimal $\mathcal{V}_{w_{opt}}$, but, in subsequent
expansion iterations, they expand
these approximations themselves rather than the optimal $\mathcal{V}_{w_{opt}}$.
In other words, only in the first expansion iteration,
both theoretical and computable optimal subspace expansion approaches
start with the {\em same} subspace and expand it. After that, at each subsequent expansion iteration, they expand {\em different} subspaces;
the theoretical optimal expansion approach
always expand the optimal subspace at that iteration, but
the computable optimal expansion approaches expand the
subspaces that {\em are already not optimally expanded} ones, so that
the subsequent expanded subspaces may constantly deviate
from the theoretical optimally ones further and further.
We have indeed observed these considerable phenomena in experiments.

In contrast, expanding $\mathcal{V}$ to $\mathcal{V}_R$ does not involve
$x$ and any other eigenvectors of $A$, and $\mathcal{V}_R$ always contains
the theoretically optimal $\mathcal{V}_{w_{opt}}$.
Therefore, $\mathcal{V}_R$ is optimal both in theory
and computations. Moreover, as we have pointed out in Remark 4 of Section 2,
unlike the a priori $\mathcal{V}_{w_{opt}}$ and its computable optimal
approximations, the subspace $\mathcal{V}_R$ can be used to compute
any other eigenpair(s) of $A$
because $\mathcal{V}_R$ itself contains the theoretical
optimally expanded subspaces $\mathcal{V}_{w_{opt}}$'s for
all the eigenvectors of $A$ at each expansion iteration.
As a result, such an expansion approach, though more
expensive than our previously proposed optimal expansion approaches,
is more robust, and may simultaneously compute more than one eigenpairs of $A$
and obtain better approximations to them.
We do not pursue this
expansion approach in the paper and leave it as future work.

\section{Numerical experiments}

Recall that, for the standard Rayleigh--Ritz and the refined
Rayleigh--Ritz methods, the Ritz vector and the refined Ritz
vector from ${\rm span}\{R_k\}$ are the computable optimal replacements of
$v_{opt}=R_kR_k^{\dagger}x/\|R_kR_k^{\dagger}x\|$
at expansion iteration $k$,
respectively, where $R_k=AV_k-V_k(V_k^HAV_k)$. All the experiments have
been performed on an Intel(R) Core(TM) i7-9700 CPU 3.00GHz
with 16 GB RAM using the Matlab R2018b with
the machine precision $\epsilon_{\rm mach}=2.22\times 10^{-16}$
under the Miscrosoft Windows 10 64-bit system.
We have used the Matlab function {\sf orth} to compute
an orthonormal basis $Q_k$ of ${\rm span}\{R_k\}$.
It is known from, e.g., \cite{jia2004,jia2001analysis,stewart2001eigensystems,vandervorst2002eigenvalue}, that
the refined Ritz vector is generally more accurate and can be
much more accurate than the Ritz vector. We have used the two
computable optimal subspace expansion approaches resulting from the
Ritz vector and the refined Ritz vector from ${\rm span}\{R_k\}$.
We shall report numerical experiments to demonstrate the effectiveness
of the theoretical optimal expansion approach and the above
two computable optimal expansion
approaches. In the meantime, we will compare them with the standard expansion
approach, i.e., the Lanczos or Arnoldi type expansion approach
with $w=v_k$, the last column of $V_k$, and the Ritz expansion
approach \cite{ye2008optimal,wu2014}, i.e., the RA method, with $w$ being the Ritz vector from $\mathcal{V}_k$ at expansion iteration $k$, respectively.

We reiterate that the theoretical optimal subspace expansion approach
cannot be used in computation since it involves the a priori $x$.
Here for a purely comparison purpose, we use it as the reference when
evaluating the effectiveness of the other subspace
expansion approaches and showing
how well the computable optimal subspace expansion approaches work. To this end,
we exploit the Matlab function {\sf eig} to
compute the desired eigenpair $(\lambda,x)$ of a general $A$ whenever it
is not given in advance, which is supposed to be
``exact". With the unit length $x$ available, we are able to
compute the error $\sin\angle(\mathcal{V}_k,x)$,
the distance between $x$ and a given $\mathcal{V}_k$.

For a given $A$, we generate $d$ vectors in a normal distribution,
orthonormalize them to obtain the
orthonormal vectors $v_1,v_2,\ldots,v_d$, and construct the
initial $d$-dimensional subspace
$$
\mathcal{V}_d={\rm span}\{v_1,v_2,\ldots,v_d\}.
$$
We then successively expand it to an $m$-dimensional $\mathcal{V}_m$ using a given
expansion approach. For $d>1$, $\mathcal{V}_d$ is non-Krylov,
so are the expanded subspaces $\mathcal{V}_k$,
$k=d+1,\ldots,m$. All the residual norms of approximate
eigenpairs $(\mu_k,\tilde{x}_k)$ mean the relative residual norms
$$
\frac{\|A \tilde{x}_k-\mu_k \tilde{x}_k\|}{\|A\|_1},\ k=d+1,\ldots,m
$$
where $\|A\|_1$ denotes the 1-norm of $A$ and $(\mu_k,\tilde{x}_k)$ is
the approximate eigenpair at iteration $k$ with $\mu_k$ being the Ritz value
and the unit length $\tilde{x}_k$ being the Ritz vector or the refined
Ritz vector when the standard Rayleigh--Ritz method or
the refined Rayleigh--Ritz method is applied to $\mathcal{V}_k,\ k=d+1,\ldots,m$.

We test two symmetric matrices and one unsymmetric matrix.
The first symmetric
$A=\diag(1,\frac{1}{2},\ldots,\frac{1}{n})$ with $n=10000$,
whose small eigenvalues are clustered. We compute the smallest
$\lambda=\frac{1}{n}$ and the corresponding eigenvector $x=e_n$,
the last column of $I$ with order $n$.
The second symmetric matrix is
the Strak\v{o}s matrix \cite[pp. XV]{meurant}, which is
diagonal with the eigenvalues labeled in the descending order:
\begin{equation}\label{relres}
  \lambda_i = \lambda_1 + \left(\frac{i-1}{n-1}\right)(\lambda_n-\lambda_1)
  \rho^{n-i},
\end{equation}
$i = 1, 2, \ldots, n$
and is extensively used to test the behavior of the Lanczos algorithm.
The parameter $\rho$ controls the eigenvalue distribution.
The large eigenvalues of $A$ are clustered for $\rho$ closer
to one and better separated for $\rho$ smaller. We compute
the largest $\lambda_1$ and the corresponding eigenvector
$x=e_1$, the first column of $I$. In the experiment,
we take $n=10000$, $\lambda_1=8$, $\lambda_n=-2$, and $\rho=0.99$.

The test unsymmetric matrix is {\sf cry2500} of $n=2500$
from the non-Hermitian Eigenvalue Problem Collection in the Matrix
Market\footnote{https://math.nist.gov/MatrixMarket}.
We are interested in the eigenvalue with the largest real part, which is clustered
with some others, and the corresponding eigenvector $x$. For these three
eigenvalue problems,
we have tested $d=5,10,15,20,25$ and expand $\mathcal{V}_d$ to $\mathcal{V}_k$
for $m=100,150, 200, 300$ when comparing the different
subspace expansion approaches.
We only report the results on $d=20$ and $m=200$ since we have observed very similar
phenomena for the other $d$'s and $m$'s.

For these three test problems, Figures~\ref{symm}--\ref{unsymm} depict the
decaying curves of the errors
\begin{equation}\label{sinxk}
\sin\angle(\mathcal{V}_k,x)
=\|(I-V_kV_k^H)x\|=\|x-V_k(V_k^Hx)\|,\ k=d+1,\ldots,m,
\end{equation}
obtained by the five subspace expansion approaches and
the residual norms (\ref{relres}) of the approximate eigenpairs obtained by
the standard Rayleigh--Ritz method and the refined Rayleigh--Ritz
method with respect to corresponding subspaces,
respectively.

Figure~\ref{symm}a, Figure~\ref{symm}c, and Figure~\ref{unsymm}a draw
$\sin\angle(\mathcal{V}_k,x)$'s for $k=d+1,\ldots,m$.
At expansion iteration $k=d,d+1,\ldots,m-1$, ``stand" denotes
the standard subspace expansion approach using $w=v_k$, ``RitzR"
and ``r-RitzR" denote
the subspace expansion approaches using the Ritz vector and
the refined Ritz vector from the corresponding ${\rm span}\{R_k\}$, respectively,
``RitzV" is the subspace expansion approach \cite{wu2014,ye2008optimal}
using the Ritz vector from the corresponding
$\mathcal{V}_k$, and ``optimal" denotes the optimal subspace expansion approach
using $R_kR_k^{\dagger}x/\|R_kR_k^{\dagger}x\|$.

Figure~\ref{symm}b, Figure~\ref{symm}d, and Figure~\ref{unsymm}b
depict relative residual norms (\ref{relres}) of approximate eigenpairs.
At expansion iteration $k=d,d+1,\ldots,m-1$,
``stand", ``RitzR'', ``RitzV", and ``optimal-R" perform the standard Rayleigh--Ritz
method on the resulting expanded subspaces $\mathcal{V}_{k+1}$ using the
subspace expansion approaches ``stand'', ``RitzR", ``RitzV", and ``optimal",
but ``r-RitzR" and ``optimal-RR" perform the refined Rayleigh--Ritz
method on the corresponding $\mathcal{V}_{k+1}$ using
the subspace expansion approaches ``r-RitzR", and ``optimal", respectively.
Particularly, we remind the reader that ``optimal-R" and ``optimal-RR" work
on the {\em same} optimally expanded subspace
$\mathcal{V}_{k+1}$ at each expansion iteration $k=d,d+1,\ldots,m-1$.

\begin{figure}[!htp]
\begin{minipage}{0.48\linewidth}
  \centerline{\includegraphics[width=4.8cm,height=4cm]{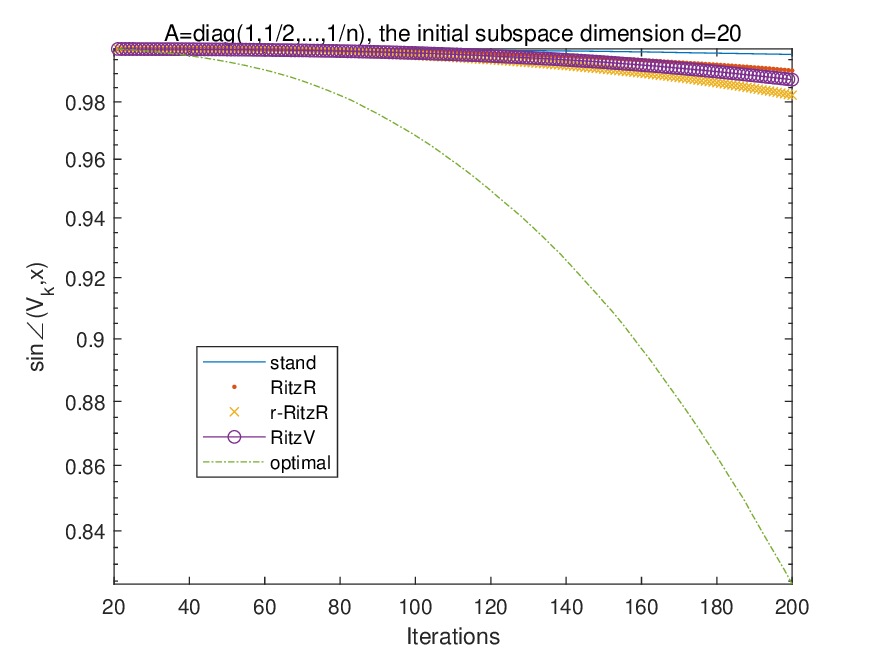}}
  \centerline{(a)}
\end{minipage}
\hfill
\begin{minipage}{0.48\linewidth}
  \centerline{\includegraphics[width=4.8cm,height=4cm]{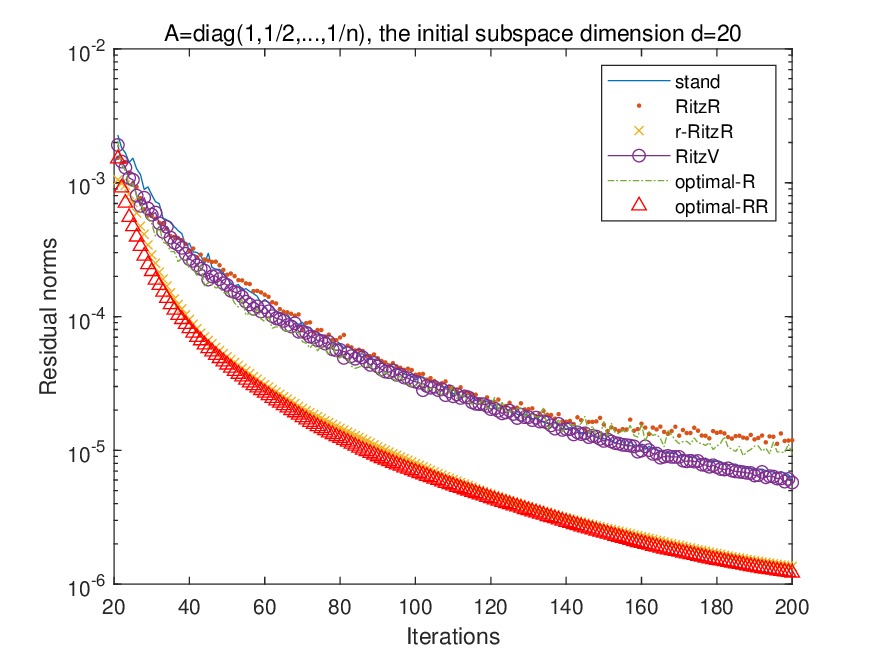}}
  \centerline{(b)}
\end{minipage}
\vfill
\begin{minipage}{0.48\linewidth}
  \centerline{\includegraphics[width=4.8cm,height=4cm]{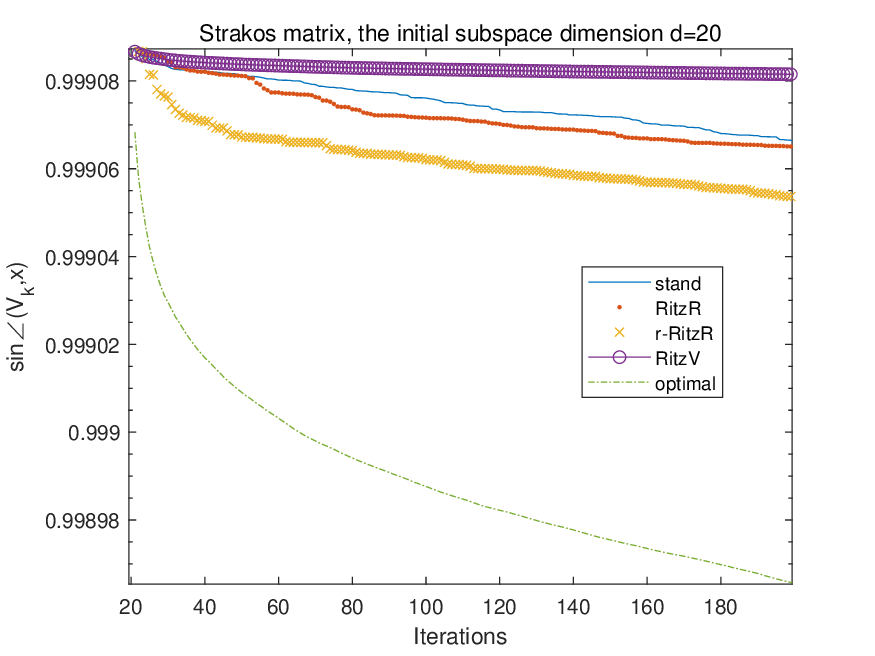}}
  \centerline{(c)}
\end{minipage}
\hfill
\begin{minipage}{0.48\linewidth}
  \centerline{\includegraphics[width=4.8cm,height=4cm]{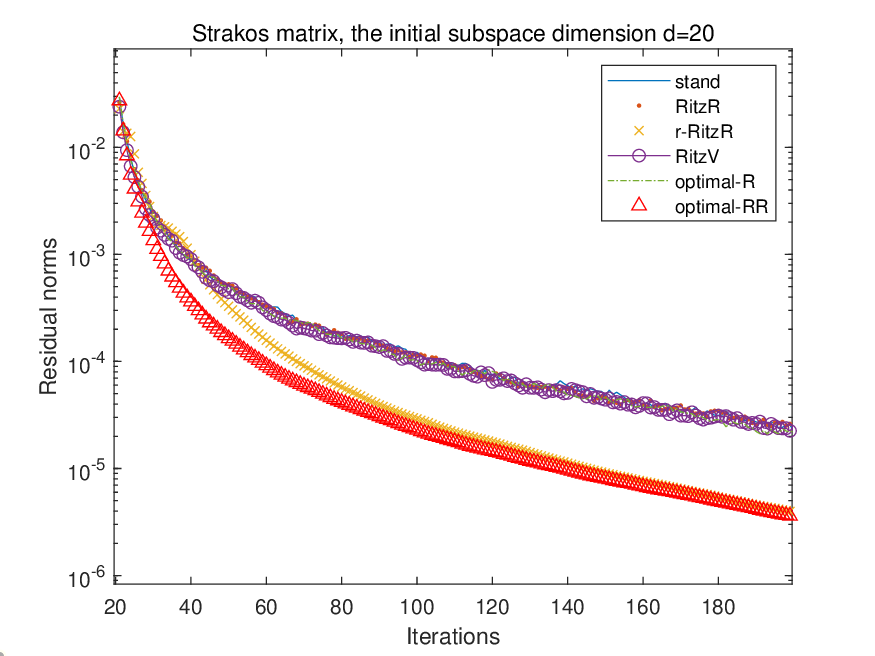}}
  \centerline{(d)}
\end{minipage}
\caption{The symmetric problems.}
\label{symm}
\end{figure}

\begin{figure}[!htp]
\begin{minipage}{0.48\linewidth}
  \centerline{\includegraphics[width=4.8cm,height=4cm]{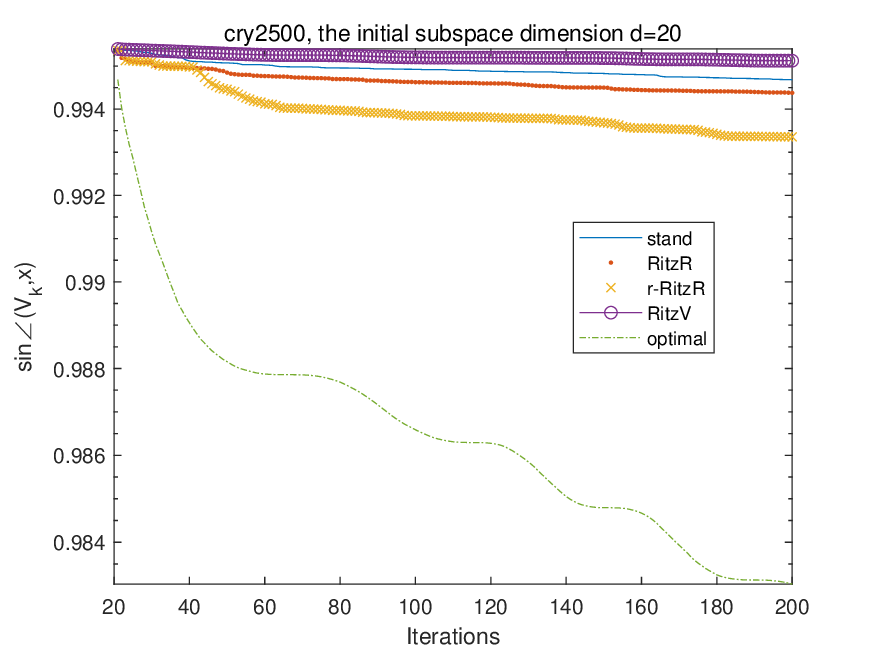}}
  \centerline{(a)}
\end{minipage}
\hfill
\begin{minipage}{0.48\linewidth}
  \centerline{\includegraphics[width=4.8cm,height=4cm]{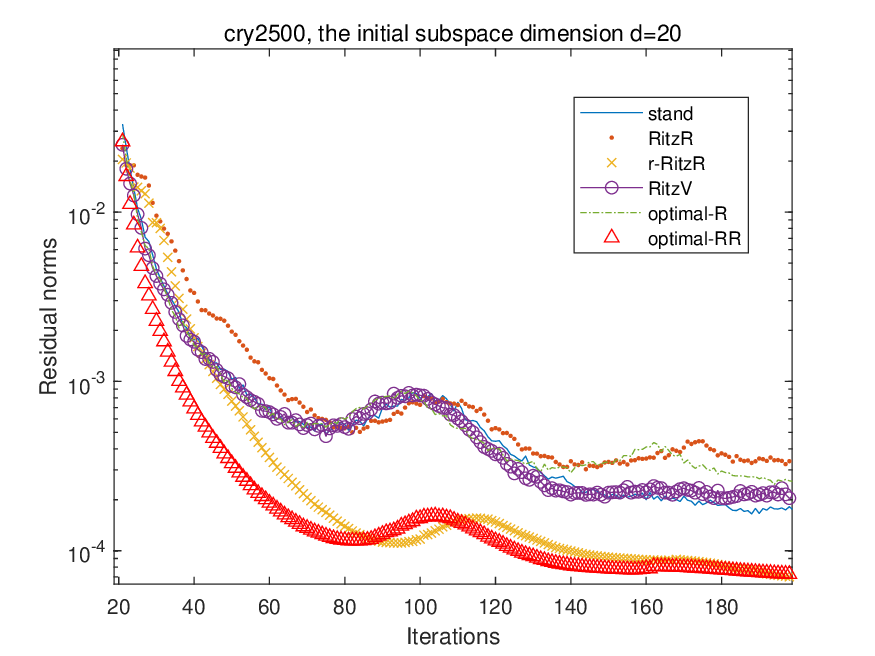}}
  \centerline{(b)}
\end{minipage}
\caption{The unsymmetric problem.} \label{unsymm}
\end{figure}

We now make several comments on the figures. Regarding
$\sin\angle(\mathcal{V}_k,x)$ for
different kinds of subspaces $\mathcal{V}_k,\ k=d+1,\ldots,m$,
we clearly observe that the optimal subspace expansion
approach ``optimal" is always the best, the computable optimal subspace
expansion approach ``r-RitzR" is the second
best, the computable optimal subspace approach ``RitzR" is the third most effective,
as is seen from Figure~\ref{symm}c and Figure~\ref{unsymm}a. They are all more
effective than the Ritz subspace expansion approach ``RitzV" and the standard
subspace subspace approach ``stand". The latter two ones
themselves win each other for the symmetric
matrices, as indicated by Figure~\ref{symm}a and Figure~\ref{symm}c,
but ``RitzV" is not as good as ``stand" for the unsymmetric matrix {\sf cry2500}.
More precisely, the computable optimal subspace expansion approach ``RitzR" is better
than ``stand" and ``RitzV" for the unsymmetric matrix and the symmetric
Strak\v{o}s matrix, but behaves very like ``RitzV" for the symmetric matrix
$A=\diag(1,1/2,\ldots,1/10000)$. On the other hand, as the figures indicate,
the three test problems are quite difficult as the desired $\lambda$ is clustered with some other
eigenvalues for each problem, causing $\sin\angle(\mathcal{V}_k,x)$ to
decay very slowly for the five subspace expansion approaches including the
optimal one. To summarize, the effectiveness of subspace expansion
approaches is, in turn, ``optimal", ``r-RitzR" and  ``RitzR",
and they are better than ``RitzV" and ``stand"; as for ``RitzV" and ``stand",
there is essentially no winner.

In contrast, the residual norms of the approximate eigenpairs exhibit more
complicated features. Note that  ``stand", ``RitzR", ``RitzV", and
``optimal-R" all use the standard Rayleigh--Ritz method on respective
subspaces $\mathcal{V}_{k+1},\ k=d,\ldots,m-1$,
to compute the Ritz approximations of $(\lambda,x)$. As we have commented in
the above and observed from the figures, $\sin\angle(\mathcal{V}_k,x)$
decays slowly for all the subspace expansion approaches, and
the residual norms of approximate
eigenpairs obtained by ``stand", ``RitzR", ``RitzV" and ``optimal-R" decrease
in a similar pattern.

However, the situation changes dramatically for ``r-RitzR"
and ``optimal-RR". First, ``r-RitzR" computes considerably
more accurate approximate eigenpairs than the afore-mentioned four methods
at each iteration. Second, ``r-RitzR" behaves very like ``optimal-RR",
the {\em ideally optimal} one,
meaning that the computable optimal ``r-RitzR" almost computes
best eigenpair approximations of $(\lambda,x)$ at each iteration. Third,
``optimal-RR" outperforms ``optimal-R" substantially, and converges
much faster than the latter, though they work on the same subspace at each
iteration. This indicates that the refined Rayleigh--Ritz method can
outperform the standard Rayleigh--Ritz method on the same subspace considerably.

In summary, our experiments on the three test problems have demonstrated that
``r-RitzR" is the best among the computable subspace expansion approaches,
and the refined Rayleigh--Ritz method on the resulting generated subspaces
behaves very like ``optimal-RR", i.e., the theoretical optimal
subspace expansion approach.

\section{Conclusions}\label{concl}

We have considered the optimal subspace expansion problem for
the general eigenvalue problem, and generalized the maximization
characterization of $\cos\angle(\mathcal{V}_w,x)$, one of Ye's two main results,
to the general case. Furthermore, we
have found the solution $b_w$ to the maximization characterization problem
of $\cos\angle(\mathcal{V}_w,x)$. Based on the expression of $b_w$, we
have analyzed the Ritz expansion approach, i.e., the RA method and argued that
the Ritz vector may not be a good approximation to the optimal
expansion vector $w_{opt}$.

Most importantly, we have
established a few results on $w_{opt}$,
$(I-P_V)Aw_{opt}$, the theoretical optimal subspace
$\mathcal{V}_{w_{opt}}$,
and $\cos\angle(\mathcal{V}_{w_{opt}},x)$ for $A$ general
and given their explicit expressions.
We have analyzed the results
in depth and obtained computable optimal replacements of
$(I-P_V)Aw_{opt}$ and
$\mathcal{V}_{w_{opt}}$ within the framework of the standard,
refined, harmonic and refined harmonic Rayleigh--Ritz methods.
Taking the standard Rayleigh--Ritz method as an example,
we have given implementation details on how to obtain
the computable optimally expanded subspace
$\mathcal{V}_{\widetilde{w}_{opt}}$ and made a cost comparison with
the subspace expansion approach with some other
expansion approaches such as the Lanczos or Arnoldi type
expansion approach with
$w=v_k$, the last column of $V_k$, the Ritz expansion approach, i.e.,
the RA method, with $w$ being the Ritz vector of $A$ from $\mathcal{V}_k$.

Our numerical experiments have
demonstrated the advantages of the computable optimal expansion
approaches over the afore-mentioned expansion approaches. Particularly, we have
observed that the computable optimal subspace expansion approach
using the refined Ritz vector of $A$ from ${\rm span}\{R_k\}$
behaves very like the theoretical optimal expansion approach
since the eigenpair approximations obtained by the former are almost as accurate
as those computed by the latter. This indicates that such computable optimal
expansion approach is as good as the theoretical optimal one and is thus
very promising when
measuring the accuracy of approximate eigenpairs in terms of residual norms.

We have also proposed a potentially more robust subspace expansion approach
that expanding $\mathcal{V}$ to $\mathcal{V}\oplus{\rm span}\{R\}$ and
computing approximate eigenpairs with respect to $\mathcal{V}\oplus{\rm span}\{R\}$
in the next iteration. For the sake of length and focus, we do not
further probe this approach in the current paper and leave it as future work.

It is well known that the harmonic Rayleigh--Ritz method is more suitable
to compute interior eigenvalues and their associated eigenvectors
than the standard Rayleigh--Ritz method
\cite{bai2000templates,stewart2001eigensystems,vandervorst2002eigenvalue}.
Furthermore, since the standard and harmonic Rayleigh--Ritz methods
may have convergence problems
when computing eigenvectors \cite{jia05,jia2001analysis},
we may gain much when using the refined Rayleigh--Ritz method
\cite{jia97,jia2001analysis} and the refined
harmonic Rayleigh--Ritz method \cite{jia05,jiali2015} correspondingly.
As a matter of fact, for a sufficiently accurate subspace
$\mathcal{V}$, the Ritz vector and the harmonic Ritz vector may be
poor approximations to $x$ and can deviate from $x$
arbitrarily (cf. \cite[pp. 284-5]{stewart2001eigensystems}
and \cite[Example 2]{jia05}), but the
refined Ritz vector and the refined harmonic Ritz vector are always
excellent approximations to $x$
\cite[pp. 291]{stewart2001eigensystems} and \cite[Example 2]{jia05},
provided that the selected approximate eigenvalues converge
to the desired eigenvalue $\lambda$ correctly. The advantages of the refined
Rayleigh--Ritz method over the standard Rayleigh--Ritz method have also
been justified by the numerical experiments in the last section.
Therefore, it is preferable to expand the subspace using
the refined Ritz vector and the refined harmonic Ritz vector
from ${\rm span}\{R_k\}$ when exterior and interior eigenpairs of $A$ are required,
respectively.

Based on the results in Section 3 and our
numerical experiments, starting with a $k$-dimensional non-Krylov subspace
with the residual $R_k$ defined by (\ref{Res}) and $\rank(R_k)\geq 2$, it is promising
to develop four kinds of more practical restarted
projection algorithms using computable
optimal subspace expansion approaches.

As it is known, there are remarkable distinctions between the Arnoldi type
expansion with $w=v_k$ and the RA method \cite{lee2007residual,leestewart07},
i.e., the Ritz expansion approach \cite{wu2014,ye2008optimal},
or the refined RA method, i.e., the refined Ritz expansion approach
\cite{wu2014} as well as between
SIRA or JD type methods and inexact SIA type methods:
A large perturbation in $Aw$ of RA type methods or in $Bw$
of SIRA and JD type methods is allowed \cite{jiali2014,jiali2015,voss},
while $Av_k$ in Arnoldi type methods or $Bv_k$ in SIA type methods must be
computed with high accuracy until the approximate
eigenpairs start to converge, and afterwards its solution accuracy
can be relaxed in a manner {\em inversely proportional} to
the residual norms of approximate
eigenpairs \cite{simonci2002,simonci2003}.
Keep these in mind. There will be a lot of important problems
and issues that deserve considerations. For example, how large errors are
allowed in $QQ^Hx/\|QQ^Hx\|$ when effectively expanding the subspace,
where the columns of $Q$ form an orthonormal basis of ${\rm span}\{R_k\}$?
This can be transformed into the problem of how large errors are
allowed in the residual $R_k$. In practical computations, this problem becomes
the one of how large errors are allowed in those computable optimal
replacements of $QQ^Hx/\|QQ^Hx\|$ within the framework of
the four kinds of projection methods, whenever perturbation errors exist
in the computable optimal replacements of $QQ^Hx/\|QQ^Hx\|$.
Is it possible to introduce the computable optimal expansion
approaches into JD and SIRA type methods?
All these need systematic exploration and analysis, and will
constitute our future work.

\section*{Acknowledgment}

I thank two referees and Professor Marlis Hochbruck very much for their numerous
valuable suggestions and comments, which inspired me to ponder the topic and
improve the presentation substantially.

\bibliographystyle{siam}

\end{document}